\documentclass{amsart}
\usepackage[margin=1in]{geometry}

\usepackage[utf8]{inputenc}
\usepackage{amsmath,amsfonts,amssymb,amsthm,tabto,color,commath,multicol,tikz-cd}
\usetikzlibrary{arrows,shapes,automata,backgrounds,petri,fit}
\usepackage{hyperref}
\hypersetup{%
 colorlinks,
    citecolor=blue,
    filecolor=blue,
    linkcolor=blue,
    urlcolor=blue,
    pdfborder = {0 0 0}}
\usepackage{multirow}
\usepackage{amsmath,amssymb,amsthm,amsfonts,latexsym,bbm,xspace,graphicx,float,mathtools,dsfont}
\usepackage{array,xypic}
\usepackage{relsize}
\usepackage{comment}
\usepackage{array} 
\usepackage{xcolor}
\usepackage[all,color]{xy}
\usepackage{comment}
\usepackage{tikz}

\newtheorem{thm}{Theorem}[section]
\newtheorem{lem}[thm]{Lemma}
\newtheorem{prop}[thm]{Proposition}
\newtheorem{coro}[thm]{Corollary}

\theoremstyle{definition}

\newtheorem{ex}[thm]{Example}

\newtheorem{rem}[thm]{Remark}

\newtheorem{defn}[thm]{Definition}

\newcommand{\cE}{\mathcal{E}}
\newcommand{\cF}{\mathcal{F}}
\renewcommand{\cL}{\mathcal{L}}
\newcommand{\cX}{\mathcal{X}}
\renewcommand{\cR}{\mathcal{R}}

\newcommand{\diamonddash}{
  \begin{tikzpicture}[scale=0.08]
    \draw (0,1) -- (1,0) -- (0,-1) -- (-1,0) -- cycle;
    \draw (-1,0) -- (1,0);
  \end{tikzpicture}
}

\let\mod\undefined
\DeclareMathOperator{\mod}{\mathrm{mod}}
\DeclareMathOperator{\rep}{\mathrm{rep}}
\DeclareMathOperator{\add}{\mathrm{add}}
\DeclareMathOperator{\mods}{\mathrm{mod}}
\DeclareMathOperator{\Hom}{\mathrm{Hom}}
\DeclareMathOperator{\Ext}{\mathrm{Ext}}
\DeclareMathOperator{\proj}{\mathrm{proj}}
\DeclareMathOperator{\inj}{\mathrm{inj}}
\let\pd\undefined
\DeclareMathOperator{\pd}{\mathrm{pd}}

\title{An exact structure approach to almost rigid modules over quivers of type $\mathbb{A}$}

\author{Thomas Brüstle, Eric J. Hanson, Sunny Roy, and Ralf Schiffler}

\address{Thomas Brüstle, Départment de Mathématiques, Université de Sherbrooke, Sherbrooke, QC, J1K 2R1, Canada and Bishop’s University, Sherbrooke, QC, J1M 1Z7, Canada}
\email{Thomas.Brustle@USherbrooke.ca, tbruestl@bishops.ca}

\address{Eric J. Hanson, Department of Mathematics, North Carolina State University, Raleigh, NC 27695, USA}
\email{ejhanso3@ncsu.edu}

\address{Sunny Roy, Départment de Mathématiques, Université de Sherbrooke, Sherbrooke, QC, J1K 2R1, Canada}
\email{Sunny.Roy@USherbrooke.ca}

\address{Ralf Schiffler, Department of Mathematics, University of Connecticut, Storrs, CT 06269-1009, USA}
\email{schiffler@math.uconn.edu}

\begin{document}

\begin{abstract}
Let $A$ be the path algebra of a quiver of Dynkin type $\mathbb{A}_n$. The module category $\textup{mod}\,A$ has a combinatorial model as the category of diagonals in a polygon $S$ with $n+1$ vertices. 
The recently introduced notion of almost rigid modules is a weakening of the classical notion of rigid modules. The importance of this new notion stems from the fact that maximal almost rigid $A$-modules are in bijection with the triangulations of the polygon $S.$ 

In this article, we give a different realization of maximal almost rigid modules. We introduce a non-standard exact structure $\cE_\diamond$
on $\textup{mod}\,A$ such that the maximal almost rigid $A$-modules in the usual exact structure are exactly the maximal rigid $A$-modules in the new exact structure. 
A maximal rigid module in this setting is the same as a tilting module. Thus the tilting theory relative to the exact structure $\cE_\diamond$ translates into a theory of maximal almost rigid modules in the usual exact structure.

As an application, we show that 
 with the exact structure $\cE_\diamond$, the module category becomes a 0-Auslander category in the sense of Gorsky, Nakaoka and Palu.

 We also discuss generalizations to quivers of type 
 $\mathbb{D}$ and gentle algebras.

\end{abstract}
\maketitle

\setcounter{tocdepth}{1}

\tableofcontents

\section{Introduction}

Let $Q$ be a quiver of Dynkin type $\mathbb{A}_n$ and $K$ a field. It is a well-known fact that if 
$\eta =  (0 \rightarrow X \rightarrow E \rightarrow Y \rightarrow 0)$
is a nonsplit short exact sequence of (right) $KQ$-modules and both $X$ and $Y$ are indecomposable, then $E$ is either indecomposable or is the direct sum of precisely two nonisomorphic indecomposables. If $E \cong E_1 \oplus E_2$ is not indecomposable, we say that $\eta$ is a \emph{diamond} short exact sequence.

A module $M$ over $KQ$ is said to be \emph{rigid} if $\Ext^1_{KQ}(M,M) = 0$. In \cite{BGMS}, the notion of rigidity is weakened by ``ignoring'' those nonsplit short exact sequences whose middle term is indecomposable, and for type $\mathbb{A}$ these are the ones that  are not diamonds. More precisely, we can restate their definition as follows. Note that it is assumed that $n \geq 3$ throughout the paper \cite{BGMS}.

\begin{defn}\cite{BGMS}\label{def:almost_rigid}
    Let $M$ be a basic (right) $KQ$-module. Then $M$ is \emph{almost rigid} if there do not exist indecomposable direct summands $X$ and $Y$ of $M$ such that $\Ext^1_{KQ}(X,Y)$ contains a diamond short exact sequence. We say that $M$ is a \emph{maximal almost rigid (MAR)} module if $M$ is almost rigid and there does not exist a nonzero module $N$ such that $M \oplus N$ is a (basic) almost rigid module.
\end{defn}

One of the main results of \cite{BGMS} shows that the maximal almost rigid modules over $KQ$ are in bijection with the set of triangulations of an $(n+1)$-gon. Moreover, this bijection arises from a geometric realization of the category $\mods KQ$ of finitely-generated (right)  modules which associates either an oriented diagonal or an oriented boundary edge of the $(n+1)$-gon to each indecomposable module. As observed in \cite[Remark~4.7]{BGMS}, this differs from the realization of the cluster category of type $\mathbb{A}_{n-2}$ established in \cite{CCS} due to the orientation, the degree shift and the inclusion of boundary edges. 

The triangulations of the $(n+1)$-gon in question also appear in \cite{reading}, where they are used as a model for the ``Cambrian lattice'' of a quiver of type $\mathbb{A}_{n-2}$. The cover relations of this lattice correspond to ``flips'' of diagonals. On the level of MAR modules, this manifests as a type of ``mutation theory'', where any pair of MAR modules related by a cover necessarily differ by only a single direct summand. Implicit in this description are the facts that every MAR module has the same number of indecomposable direct summands and that the modules corresponding to boundary edges are direct summands of every MAR modules, each of which is established in \cite{BGMS}. These generalize well-known properties from classical tilting theory, namely that every tilting module has the same number of direct summands and that the projective-injective modules are direct summands of every tilting module.

The goal of this paper is to interpret the results described above using the theory of exact categories, and more specifically that of 0-Auslander categories\footnote{Note that 0-Auslander categories are actually defined as certain examples of \emph{extriangulated} categories. Extriangulated categories include both exact and triangulated categories as special cases.} established in the recent paper \cite{GNP}.

The paper is outlined as follows. In Section~\ref{sec:exact}, we recall background information about exact categories, 0-Auslander categories, and the associated notions of rigidity and tilting. We then recall technical details regarding the Auslander-Reiten quivers in type $\mathbb{A}_n$ in Section~\ref{sec:AR_quivers}. These well-known facts will be necessary in proving our main results later in the paper. In Section~\ref{sec:MAR}, we introduce an exact structure $\cE_\diamond$. We then prove our main results.

\begin{thm}\label{thm:intro}\textnormal{[{Proposition~\ref{prop:diamond_seqs}, Corollary~\ref{cor:almost rigid},  Corollary~\ref{cor:new}, Theorem~\ref{thm:0_auslander}, and Corollary~\ref{cor:tilting}}]}
    Let $Q$ be a quiver of type $\mathbb{A}_n$. Then the following hold.
    \begin{enumerate}
        \item Let $\eta = (0 \rightarrow X \rightarrow E \rightarrow Y \rightarrow 0)$ be a nonsplit short exact sequence with $X$ and $Y$ indecomposable. Then $\eta \in \cE_\diamond(Y,X) \subseteq \Ext^1_{KQ}(Y,X)$ if and only if $\eta$ is a diamond short exact sequence.
        \item A basic module $M \in \mods KQ$ is almost rigid if and only if it is $\cE_\diamond$-rigid.
        \item A basic module $M \in \mods KQ$ is maximal almost rigid if and only if it is $\cE_\diamond$-tilting.
        \item An indecomposable module $X \in \mods KQ$ is $\cE_\diamond$-projective if and only if either $X$ is projective or the Auslander-Reiten sequence $0 \rightarrow \tau X \rightarrow E \rightarrow X \rightarrow 0$ has $E$ indecomposable.
        \item An indecomposable module $X \in \mods KQ$ is $\cE_\diamond$-injective if and only if either $X$ is injective or the Auslander-Reiten sequence $0 \rightarrow X \rightarrow E \rightarrow \tau^{-1} X \rightarrow 0$ has $E$ indecomposable.
        \item $(\mods KQ, \cE_\diamond)$ is a 0-Auslander category.
    \end{enumerate}
\end{thm}

As we discuss in Section~\ref{sec:exact}, the theory of tilting in 0-Auslander categories thus gives an alternate explanation for why MAR modules always have the same number of direct summands and always contain the modules corresponding to the boundary edges of the $(n+1)$-gon.

We discuss possible generalizations in Section~\ref{sect 5}. In Section~\ref{sec:typeD},  we give examples showing why the theory developed in this paper cannot be readily extended to quivers of Dynkin type $\mathbb{D}$. In  Section~\ref{sec:gentle} we give an example of a possible generalization of our results to gentle algebras. In that example conditions (1)-(5) of Theorem 1.2 still hold, but condition (6) doesn't.

\subsection*{Acknowledgements}
TB was supported by NSERC Discovery Grant RGPIN/04465-2019.
EJH was partially supported by Canada Research Chairs CRC-2021-00120, NSERC Discovery Grants RGPIN-2022-03960 and RGPIN/04465-2019, and an AMS-Simons travel grant. A portion of this work was completed while EJH was a postdoctoral fellow at l'Universit\'e du Qu\'ebec \`a Montr\'eal and l'Universit\'e de Sherbrooke.
SR is supported by the B2X FRQNT grant. RS was supported by the NSF grants  DMS-2054561 and DMS-2348909. The authors would like to thank the Isaac Newton Institute for Mathematical Sciences for support and hospitality during the programme Cluster Algebras and Representation Theory where work on this paper was undertaken. This work was supported by: EPSRC Grant Number EP/R014604/1. 

\section{Exact structures and 0-Auslander categories}\label{sec:exact}

In this section, we recall relevant background information on exact structures and 0-Auslander categories. We conclude with a discussion of the mutation theory for maximal rigid objects in 0-Auslander categories as developed in \cite[Section~4]{GNP}. Since we are mostly interested in applying these results to an exact structure on the category $\mods \Lambda$ of finitely-generated (right) modules over a finite-dimensional algebra $\Lambda$, we restrict our exposition to this case. We note, however, that both exact structures/categories and 0-Auslander categories are defined much more generally.

\subsection{Exact structures}

Exact structures have been introduced by Quillen 
\cite{quillen} as an axiomatic framework that allows to use methods from homological algebra in a  context relative to a fixed class of short exact sequences. Besides the axiomatic approach, there are a number of ways to describe an exact structure  $\cE$. We will describe an exact structure as a certain subfunctor $\cE$ of the bifunctor $\Ext_\Lambda^1$, or by specifying what are the projective objects relative to  $\cE$, or by deciding which Auslander-Reiten sequences belong to $\cE$.

Let $\cE$ be a class of short exact sequences in $\mods \Lambda$. We assume that $\cE$ is closed under isomorphisms. When a short exact sequence $\eta = (0 \to A \xrightarrow{f} E \xrightarrow{g} B \to 0)$
belongs to $\cE$, we say that $\eta$ is an ($\cE$-)\emph{admissible short exact sequence}, and also that $f$ is an ($\cE$-)\emph{admissible monomorphism} and $g$ is an ($\cE$-)\emph{admissible epimorphism}. The closure under isomorphisms then implies that the classes of admissible short exact sequences, admissible monomorphisms, and admissible epimorphisms uniquely determine one another.
Quillen's definition can then be rephrased as follows:
\begin{defn}\label{def:exact}
A class $\cE$  of short exact sequences in $\mods \Lambda$ is said to be an \emph{exact structure} on $\mods\Lambda$ if all of the following hold:
 \begin{enumerate}
	\item  $\cE$ contains all split exact sequences.
        \item $\cE$ is closed under compositions of admissible monomorphisms, that is, if $f: X \to Y$ and $g: Y \to Z$ are admissible monomorphisms, then  $g \circ f : X \to Z$ is also an admissible monomorphism.
        Likewise, $\cE$ is closed under compositions of admissible epimorphisms.
        \item $\cE$ is closed under pushouts: If $f: X \to Y$ is an admissible monomorphism,
        and $X\xrightarrow{h} W$ is any morphism in $\mods \Lambda$, then the pushout of $f$ along $h$ yields a short exact sequence in $\cE$.
        Likewise, $\cE$ is closed under pullbacks.
	\end{enumerate}
\end{defn}

\begin{rem}
    In the literature, it is common both to refer to $\cE$ as an \emph{exact structure} on $\mods\Lambda$ and to refer to the pair $(\mods \Lambda,\cE)$ as an \emph{exact category}.
\end{rem}

Quillen's original definition contained also a fourth (so-called ``obscure'' axiom) which was shown in \cite{keller} to be superfluous.
It has been observed in \cite{AS} that in the context of module categories, the axioms (1) and (3) together can be rephrased as $\cE$ forming an additive subfunctor $\cE(-,-)$ of the bifunctor $\Ext^1_\Lambda(-,-)$. Here (1) states that $\cE$ contains the zero elements of all groups $\Ext^1(X,Y)$ (under the Baer sum), and axiom (3) corresponds to functoriality. An additive subfunctor of $\Ext^1_\Lambda(-,-)$ satisfying condition (2) has been called a {\em closed} subfunctor of $\Ext^1_\Lambda(-,-)$, and \cite{DRSS} show that these correspond precisely to the exact structures on $\mods\Lambda.$

For $\Lambda$ representation-finite, it has been shown in \cite{En18} (see also \cite[Theorem 5.7]{BHLR}) that an exact structure on $\mods\Lambda$ is uniquely determined by the Auslander-Reiten sequences it contains, and in fact there is a bijection between exact structures on $\mod \Lambda$ and sets of (isomorphism classes of) Auslander-Reiten sequences in $\mods\Lambda$.
One can also characterize an exact structure by the set of its $\cE$-projective (resp. $\cE$-injective) objects, defined as follows.

\begin{defn}\label{def:projectives}
Let $\cE$ be an exact structure on $\mods\Lambda$.
 \begin{enumerate}
     \item We say an object $P \in \mods\Lambda$ is \emph{$\cE$-projective}, in symbols $P \in \proj(\cE)$, if $g_* = \Hom_\Lambda(P,g)$ is an epimorphism for all $Y, Z \in \mods\Lambda$ and for all admissible epimorphisms $g: Y \to Z$. Equivalently, $P \in \proj(\cE)$ if and only if $\Hom_\Lambda(P,-)$ sends admissible exact sequences to exact sequences.
     \item We say an object $I \in \mods\Lambda$ is \emph{$\cE$-injective}, in symbols $I \in \inj(\cE)$, if $f_* = \Hom_\Lambda(f,I)$ is an epimorphism for all $X, Y \in \mods\Lambda$ and for all admissible monomorphisms $f: X \to Y$. Equivalently, $I \in \inj(\cE)$ if and only if $\Hom_\Lambda(-,I)$ sends admissible exact sequences to exact sequences.
 \end{enumerate}
\end{defn}

The axiomatic setup of an exact structure allows one to introduce an $\cE$-relative version for many concepts from homological algebra, like the existence of a long homology sequence. In particular, the fact that $\Hom_\Lambda(P,-)$ is an exact functor on admissible exact sequences can be rephrased by saying that $\cE(P,-)$ is zero. Therefore $P \in \mods\Lambda$ is an $\mathcal{E}$-projective object precisely when there is no non-split short exact sequence $(0\rightarrow A\rightarrow B \rightarrow P\rightarrow 0) \in \mathcal{E}$. The dual characterization likewise holds for $\cE$-injective objects.

The following definition from \cite{AS} uses this characterization to describe an exact structure in terms of its set of projective objects.

\begin{defn}  
Let $\mathcal{X}$ be 
a subcategory 
of $\mods\Lambda$. We define 
\[\cF_\mathcal{X} = \{ \eta= (0\rightarrow A\rightarrow B \rightarrow C\rightarrow 0) \;  | \;\Hom_\Lambda(X,-)  \mbox{ is an exact functor on $\eta$ for all } X \in \mathcal{X} \}. \]
\end{defn}

\begin{rem}\label{rem:FX}
    Since covariant hom-functors are always left exact, an exact sequence
    $$0 \rightarrow A \rightarrow B \xrightarrow{f} C \rightarrow 0$$
    is $\cF_\cX$-exact if and only if $\Hom(X,f)$ is surjective for all $X \in \cX$.
\end{rem} 

In both the following proposition and the remainder of the paper, we denote by $\proj(\Lambda)$ (resp. $\inj(\Lambda)$) the full subcategory of $\mods\Lambda$ consisting of those objects which satisfy the usual notion of projectivity (resp. injectivity) in $\mods\Lambda$. These can be seen as the ``true'' projectives and injectives, as opposed to the ``relative'' projectives and injectives. We also denote by $\tau$ the Auslander-Reiten translation. For a subcategory $\cX$, we denote by $\tau \cX$ the full subcategory of objects of the form $\tau X$ for $X \in \cX$.

\begin{prop}\label{prop:exact_structure_proj} \cite[Section~1]{AS}
Let $\mathcal{X}$ be 
a subcategory of $\mod \Lambda$. Then $\cF_{\mathcal{X}}$ is an exact structure. Moreover, $\proj(\cF_\mathcal{X}) = \add(\mathcal{X}	\cup \proj(\Lambda))$ and $\inj(\cF_\mathcal{X}) = \add(\tau \cX \cup \inj(\Lambda))$.
\end{prop}

\subsection{Relative homological dimensions and relative tilting}

The notions of $\mathcal{E}$-projectives and $\mathcal{E}$-injectives allow one to define relative versions of many  of the dimensions studied in homological algebra.

\begin{defn}\label{def:proj_dim}
Let $\mathcal{E}$ be an exact structure on $\mods\Lambda$.
\begin{enumerate}
    \item We say that $\cE$ \emph{has enough projectives} if for every $M \in \mods\Lambda$ there exists an $\cE$-admissible epimorphism $f_0: P_0 \rightarrow M$ with $P \in \proj(\cE)$.
    \item Suppose $\cE$ has enough projectives and let $M \in \mods\Lambda$. We say an exact sequence
    \[\xymatrix{\cdots \;  P_2 \ar@{->}[r]^{f_2} & P_1 \ar@{->}[r]^{f_1} & P_0  \ar@{->}[r]^{f_0} & M  \ar@{->}[r] &0}\]
     is an $\mathcal{E}$-projective resolution of $M$ 
    if all of $P_i$ are $\mathcal{E}$-projective and each short exact sequence
    $$0 \rightarrow \ker f_i \rightarrow P_i \rightarrow \mathrm{im} f_i \rightarrow 0$$
    is $\cE$-admissible.  The \emph{length} of a $\cE$-projective resolution is the smallest index $i$ for which $P_i = 0$ (or infinity if no such $i$ exists).
    \item Suppose $\cE$ has enough projectives and let $M \in \mods\Lambda$. The \emph{$\cE$-projective dimension} of $M$, denoted $\pd_\cE(M)$, is the minimal length among all $\cE$-projective resolutions of $M$.
    \item Suppose $\cE$ has enough projectives. The \emph{$\cE$-global dimension} of $\Lambda$ is $\sup_{M \in \mods\Lambda} \pd_\cE(M)$.
\end{enumerate}
\end{defn}

\begin{rem}
    \begin{enumerate}
\item In later sections, we restrict to the case $\Lambda = KQ$ for $Q$ a quiver of type $\mathbb{A}_n$. This algebra is representation finite, and thus every exact structure of $\mods KQ$ has enough projectives.
 \item     Note that the $\cE$-projective dimension of a module $M$ need not be related to the (classical) projective dimension $\pd(M)$: one may have $\pd_\cE(M) < \pd(M)$ since there are more $\cE$-projective modules than projectives. For example, each $\cE$-projective $M$ which is not in $\proj \Lambda$ has $0 = \pd_\cE(M) < \pd(M)$. 
 On the other hand, an $\cE$-projective resolution requires all maps be composed of $\cE$-admissible maps, so projective resolutions are not necessarily $\cE$-projective resolutions. 
 In fact, we give in Section~\ref{sec:typeD} an example of a module $M$ over a quiver algebra (which is always of global dimension one) whose $\cE$-projective dimension is two.
    \end{enumerate}
  \end{rem}

We now recall the following definition.

\begin{defn}\cite{AS_2}
Let $\cE$ be an exact structure on $\mods\Lambda$ with enough projectives. A module $T \in \mods\Lambda$ is said to be a \emph{$\cE$-tilting module} if the following statements all hold:
\begin{enumerate}
\item $\pd_{\cE}(T) \leq 1$.
\item For all $P \in \proj(\cE)$, there is an $\cE$-admissible short exact sequence $0\rightarrow P\rightarrow T_1 \rightarrow T_2\rightarrow 0$  with $T_1,T_2 \in \add(T)$.
\item $T$ is \emph{$\cE$-rigid}; that is, $\cE(T,T) = 0$.
\end{enumerate}
\end{defn}

\begin{rem}
    Note that \cite{AS_2} more generally defines a notion of $\cE$-tilting for modules of global dimension $n$ ($n \in \mathbb{N}$). As is standard in classical tilting theory, this can be clarified by referring to $\cE$-tilting modules of projective dimension $n$ as ``$n$-$\cE$-tilting'' (see e.g. \cite{wei}). One can then drop the subscript when dealing only with $n = 1$, resulting in the convention of the present paper and e.g. \cite{MM}.
  
    Note also that both \cite{GNP} and \cite{palu} refer to the notion of $\cE$-tilting used in this paper as just ``tilting''. Both papers also consider a different (but equivalent) notion of ``$\mathbb{E}$-tilting''.
\end{rem}

In Section~\ref{sec:MAR}, we will show that the maximal almost rigid modules in $\mods KQ$ ($Q$ of type $\mathbb{A}_n$) are precisely the basic $\cE$-tilting modules for some exact structure $\cE$. In order to do so, it will be useful to relate the notion of $\cE$-tilting to that of ``maximal $\cE$-rigidity''. We do this using the theory of 0-Auslander categories from \cite{GNP}, the definition of which we recall now. We use the simplified version of the definition appearing as \cite[Definition~5.1]{palu}. Recall, however, that \cite{palu} works with arbitrary extriangulated categories, while we work only with exact structures on $\mods\Lambda$.

\begin{defn}\label{def:tilting}
    Let $\cE$ be an exact structure on $\mods\Lambda$ with enough projectives. Then $(\mods\Lambda, \cE)$ is a \emph{0-Auslander category} if the following both hold.
    \begin{enumerate}
        \item The $\cE$-global dimension of $ \Lambda$ is at most 1.
        \item For all $P \in \proj(\cE)$ there exists an $\cE$-admissible short exact sequence $0 \rightarrow P \rightarrow Q \rightarrow I \rightarrow 0$ with $Q \in \proj(\cE) \cap \inj(\cE)$ and $I \in \inj(\cE)$. That is, the \emph{$\cE$-dominant dimension} of $\Lambda$ is at most 1.
    \end{enumerate}
\end{defn}

We now recall the following from \cite[Proposition~5.7]{palu} (see also \cite[Theorem~4.3]{GNP}). For a module $M \in \mods\Lambda$, we denote by $|M|$ the number of isomorphism classes of indecomposable direct summands of~$M$.

\begin{thm}\label{thm:tilting_char}
    Let $\cE$ be an exact structure on $\mods\Lambda$ with enough projectives, and suppose that $(\mods\Lambda,\cE)$ is a 0-Auslander category. Suppose further that there exists $P \in \mods(\Lambda)$ such that $\proj(\cE) = \add(P)$. 
    Then the following are equivalent for a module $T \in \mods\Lambda$.
    \begin{enumerate}
        \item $T$ is $\cE$-tilting.
        \item $T$ is \emph{maximal $\cE$-rigid}; that is, $T$ is $\cE$-rigid and if $T \oplus X$ is $\cE$-rigid, then $X \in \add(T)$.
        \item $T$ is \emph{complete $\cE$-rigid}; that is, $T$ is $\cE$-rigid and $|T| = |P|$.
    \end{enumerate}
\end{thm}


\section{Combinatorics of Auslander-Reiten quivers of type $\mathbb{A}$}\label{sec:AR_quivers}

In this section, we recall background information on Auslander-Reiten quivers of path algebras of type $\mathbb{A}_n$. The results of this section are well-known, and we refer to \cite[Section~3.1]{schiffler} for further details.

Fix for the duration of this section a quiver $Q$ of type $\mathbb{A}_n$. We freely move between the language of ($K$-)representations of $Q$ and of (right) $KQ$-modules. We adopt the standard convention of indexing the vertices of $Q$ by $[n] = \{1,\ldots,n\}$ so that the arrows of $Q$ are always of the form $i \rightarrow i+1$ or $i \leftarrow i+1$. For $i \in [n]$, we denote by $P(i)$ and $I(i)$ the projective cover and injective envelope of the simple representation $S(i)$ supported at the vertex $i$. In examples, we denote other representations by their radical filtrations.

The indecomposable representations of $Q$ can be partitioned into sets $\cL_1,\ldots,\cL_n$ so that a representation $M$ lies in $\cL_i$ if and only if $M \cong \tau^{-m} P(i)$ for some $m$. There is then an embedding of the Auslander-Reiten quiver $\Gamma(Q)$ into $\mathbb{Z} \times [n] \subseteq \mathbb{R}^2$ which satisfies the following:
\begin{itemize}
    \item The representations in $\cL_i$ all lie on the horizontal line $y = i$ (for all $i \in [n]$).
    \item Every irreducible morphism is represented by an arrow pointing in the direction $(1,1)$ or $(1,-1)$.
\end{itemize}
When we refer to $\Gamma(Q)$, we always refer to the Auslander-Reiten quiver with this embedding. An example is shown in Figure~\ref{fig:AR_ex}. Note that the embedding $\Gamma(Q)$ is unique up to horizontal shift.

\begin{figure}
    \begin{tikzcd}[column sep = 1em]
        5 = P(5)\arrow[dr] && 4 \arrow[dr] && {\footnotesize \begin{matrix}3\\2\end{matrix}} \arrow[dr] && 1 = I(1)\\
        & {\small \begin{matrix}4\\5\end{matrix}} = P(4) \arrow[ur]\arrow[dr] && {\footnotesize\begin{matrix}3\\2 \ 4\end{matrix}} \arrow[ur]\arrow[dr] && {\footnotesize \begin{matrix}1 \ 3\\2\end{matrix}} = I(2) \arrow[ur]\arrow[dr]\\
        &&{\scriptsize \begin{matrix}3\\ 2 \ 4\\\phantom{2 \ \ }5\end{matrix}}  = P(3)\arrow[ur]\arrow[dr] && {\footnotesize \begin{matrix}1 \ 3 \ \\ \ \ \ 2 \ 4\end{matrix}} \arrow[ur]\arrow[dr] && 3 = I(3)\\
        & 2 = P(2)\arrow[ur]\arrow[dr] && {\scriptsize \begin{matrix}1 \ 3 \ \\ \ \ \ 2 \ 4\\ \phantom{11111}5\end{matrix}}\arrow[ur]\arrow[dr] && {\footnotesize \begin{matrix}3\\4\end{matrix}} = I(4)\arrow[ur]\\
        && {\footnotesize \begin{matrix}1\\2\end{matrix}} = P(1)\arrow[ur] && {\scriptsize \begin{matrix}3\\4\\5\end{matrix}} = I(5) \arrow[ur]
    \end{tikzcd}
    \caption{The embedded AR quiver $\Gamma(Q)$ for $Q = 1\rightarrow 2 \leftarrow 3 \rightarrow 4 \rightarrow 5$.}\label{fig:AR_ex}
\end{figure}
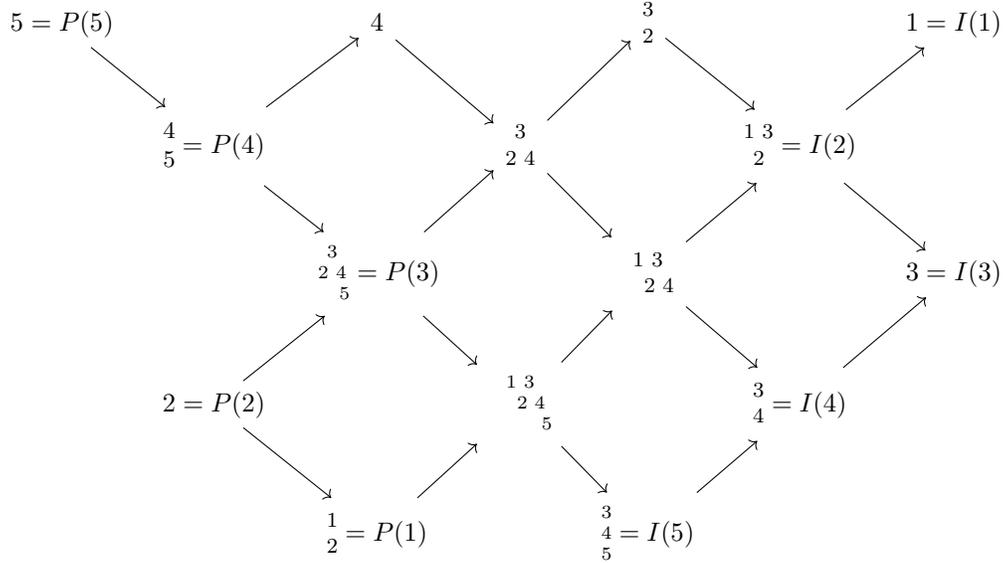

Let $M$ be an indecomposable representation. We denote by $\Sigma_{\nearrow}(M)$ and $\Sigma_{\searrow}(M)$ the rays (in $\mathbb{R}^2$) of slope 1 and $-1$ with base point corresponding to $M \in \Gamma(Q)$. Similarly, we denote by $\Sigma_{\swarrow}(M)$ and $\Sigma_{\nwarrow}(M)$ the corays of slope 1 and $-1$ with base point corresponding to $M \in \Gamma(Q)$. We say that an indecomposable representation $N$ (seen as a vertex of $\Gamma(Q)$) lies \emph{at the end of} e.g. $\Sigma_{\nearrow}(M)$ if $N$ lies on $\Sigma_{\nearrow}(M)$ and there does not exist $L \neq N$ such that $L$ lies on $\Sigma_{\nearrow}(N)$. 
For example, in Figure~\ref{fig:AR_ex}, $I(3), I(4)$, and $I(5)$ lie on $\Sigma_{\nearrow}(I(5))$, with $I(3)$ lying at the end. Note that, in the language of \cite[Section~3.1.4.1]{schiffler}, we have that there is a sectional path from $M$ to $N$ if and only if $N$ lies on $\Sigma_{\nearrow}(M) \cup \Sigma_{\searrow}(M)$.

We denote by $\cR_\rightarrow (M)$ the set of indecomposable representations whose position in $\Gamma(Q)$ lies in the slanted rectangular region whose left boundary is $\Sigma_{\nearrow}(M) \cup \Sigma_{\searrow}(M)$. For example, in Figure~\ref{fig:AR_ex}, we have
$$\cR_\rightarrow \left({\footnotesize \begin{matrix}1 \ 3 \ \\ \ \ \ 2 \ 4\end{matrix}}\right) = \left\{{\footnotesize \begin{matrix}1 \ 3 \ \\ \ \ \ 2 \ 4\end{matrix}}, I(1), I(2), I(3), I(4)\right\}.$$
We symmetrically denote by $\cR_\leftarrow(M)$ the set of indecomposable representations whose position in $\Gamma(Q)$ lies in the slanted rectangular region whose right boundary is $\Sigma_{\swarrow}(M) \cup \Sigma_{\nwarrow}(M)$. See \cite[Section~3.1.4.1]{schiffler} for additional examples.

The following terminology is useful for characterizing the existence of short exact sequences and nonzero morphisms.

\begin{defn}\label{defn:rects}
    Let $M$ and $N$ be indecomposable representations.
   \begin{enumerate}
        \item We say that $(M,N)$ is a \emph{rectangular pair} if the following both hold.
         \begin{enumerate}
                \item The ray $\Sigma_{\nearrow}(M)$ and coray $\Sigma_{\nwarrow}(N)$ intersect and there is a module $E_1$ at $\Sigma_{\nearrow}(M)\cap \Sigma_{\nwarrow}(N)$.
            \item The ray $\Sigma_{\searrow}(M)$ and coray $\Sigma_{\swarrow}(N)$ intersect and there is a module $E_2$ at $\Sigma_{\searrow}(M)\cap \Sigma_{\swarrow}(N)$.
        \end{enumerate}
        \item Suppose $(M,N)$ is a rectangular pair. We say that $(M,N)$ is \emph{degenerate} if $N$ lies on $\Sigma_{\nearrow}(M) \cup \Sigma_{\searrow}(M)$. Otherwise, we say that $(M,N)$ is \emph{non-degenerate}.
        \item We say that $(M,N)$ is a \emph{lower deleted pair} if the following both hold.
        \begin{enumerate}
            \item The ray $\Sigma_{\nearrow}(M)$ and coray $\Sigma_{\nwarrow}(N)$ intersect and there is a module $E$ at $\Sigma_{\nearrow}(M)\cap \Sigma_{\nwarrow}(N)$.
            \item Let $L$ be the module at the end of the ray $\Sigma_{\searrow}(M)$. Then $\tau^{-1} L$ is the module at the end of the coray $\Sigma_{\swarrow}(N)$.
        \end{enumerate}
        \item We say that $(M,N)$ is a \emph{upper deleted pair} if the following both hold.
        \begin{enumerate}
            \item The ray $\Sigma_{\searrow}(M)$ and coray $\Sigma_{\swarrow}(N)$ intersect and there is a module $E$ at $\Sigma_{\searrow}(M)\cap \Sigma_{\swarrow}(N)$.
            \item Let $L$ be the module at the end of the ray $\Sigma_{\nearrow}(M)$. Then $\tau^{-1} L$ is the module at the end of the coray $\Sigma_{\nwarrow}(N)$.
        \end{enumerate}
    \end{enumerate}
\end{defn}

Schematic diagrams of an upper deleted pair and of a non-degenerate rectangular pair are shown in Figure~\ref{fig:seq_ex}. Note that the shape of an upper (resp. lower) deleted pair is that of a slanted rectangle with a single endpoint deleted.
 We defer specific examples to Example~\ref{ex:homs_and_exts}.
 
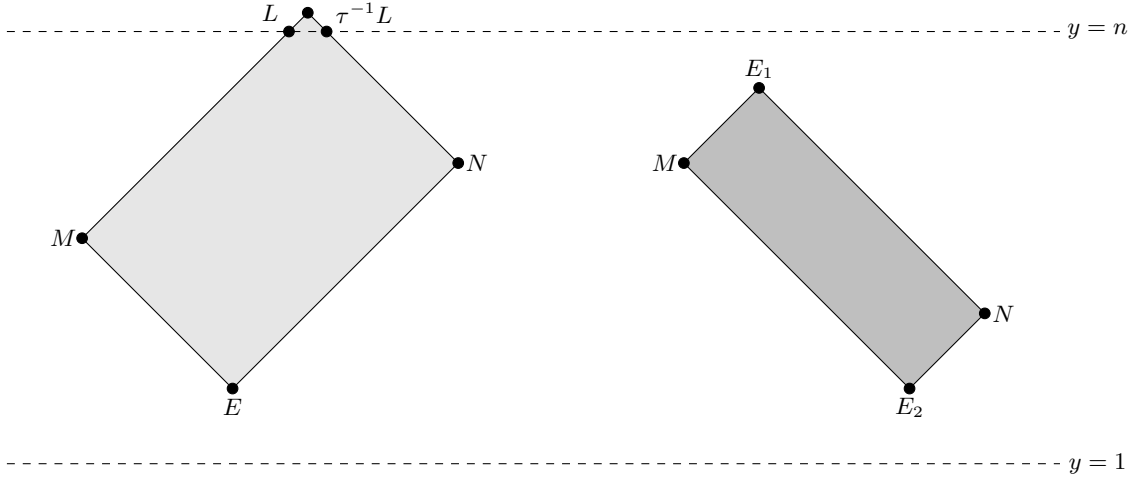
\begin{figure}
{\small
\[
\begin{tikzpicture}
[inner sep=0.5mm]
\filldraw[fill=gray!20!white, draw=black] (0,0) -- (3,3) -- (5,1) -- (2,-2) -- cycle;
 \draw [dashed] (-1,2.75) -- (13,2.75);
\draw [dashed] (-1,-3) -- (13,-3);
\node at (13.5,2.75) {$y = n$};
\node at (13.5,-3) {$y = 1$};
\node at (-0.25,0) {$M$};
 \node at (0,0) [shape=circle,draw,fill] {};
 \node at (5,1) [shape=circle,draw,fill] {};
 \node at (3,3) [shape=circle,draw,fill] {};
  \node at (2.75,2.75) [shape=circle,draw,fill] {};
  \node at (2.5,3) {$L$};
  \node at (3.75,3) {$\tau^{-1}L$};
   \node at (3.25,2.75) 
   [shape=circle,draw,fill] {};
 \node at (2,-2) [shape=circle,draw,fill] {};
 \node at (2,-2.25) {$E$};
 \node at (5.25,1) {$N$};

 \filldraw[fill=gray!50!white, draw=black] (8,1) -- (9,2) -- (12,-1) -- (11,-2) -- cycle;
\node at (7.75,1) {$M$};
 \node at (8,1) [shape=circle,draw,fill] {};
 \node at (9,2) [shape=circle,draw,fill] {};
 \node at (12,-1) [shape=circle,draw,fill] {};
  \node at (11,-2) [shape=circle,draw,fill] {};
  \node at (9,2.25) {$E_1$};
   \node at (11,-2.25) {$E_2$};
 \node at (12.25,-1) {$N$};
  \end{tikzpicture}
\]}
\caption{Schematic diagrams of an upper deleted (light gray) and of a non-degenerate rectangular pair (dark gray).}\label{fig:seq_ex}
\end{figure}

We now recall how one obtains the structure of morphisms and short exact sequences from the quiver $\Gamma(Q)$.

\begin{lem}\label{lem:hom}\cite[Sec.~3.1.4.1]{schiffler}
    Let $M$ and $N$ be indecomposable representations. Then $\Hom(M,N) \cong K$ if and only if the following equivalent conditions hold. Otherwise, $\Hom(M,N) = 0$.
    \begin{enumerate}
        \item $(M,N)$ is a rectangular pair.
        \item $N \in \mathcal{R}_{\rightarrow}(M)$.
        \item $M \in \mathcal{R}_{\leftarrow}(N)$.
    \end{enumerate}
\end{lem}

As a consequence of Lemma~\ref{lem:hom}, we obtain the following.

\begin{prop}\label{prop:factor}
    Let $M, N$ and $L$ be indecomposable representations and suppose there are nonzero morphisms $f: M \rightarrow N$ and $g: N \rightarrow L$. Then every morphism $M \rightarrow L$ factors through both $f$ and $g$.
\end{prop}

The next result explains the structure of short exact sequences.

\begin{lem}\label{lem:ext}\cite[Section~3.1.4.3]{schiffler}\textnormal{ (see also \cite[Proposition~6.5]{BGMS})} Let $M$ and $N$ be indecomposable representations.
\begin{enumerate}
    \item Suppose that $(M,N)$ is a non-degenerate rectangular pair. Let $E_1$ be the module at $\Sigma_{\nearrow}(M)\cap\Sigma_{\nwarrow}(N)$ and let $E_2$ be the module at $\Sigma_{\searrow}(M) \cap \Sigma_{\swarrow}(N)$. Then $\Ext^1(N,M) \cong K$ and is spanned by a (diamond) short exact sequence
    $$0 \rightarrow M \rightarrow E_1 \oplus E_2 \rightarrow N \rightarrow 0.$$
    \item Suppose that $(M,N)$ is an upper or lower deleted pair. Let $E$ be the module at $\Sigma_{\nearrow}(M)\cap\Sigma_{\nwarrow}(N)$ (lower case) or $\Sigma_{\searrow}(M) \cap \Sigma_{\swarrow}(N)$ (upper case). Then $\Ext^1(N,M) \cong K$ and is spanned by a short exact sequence
    $$0 \rightarrow M \rightarrow E \rightarrow N \rightarrow 0.$$
    \item Suppose that $(M,N)$ is not a non-degenerate rectangular pair, an upper deleted pair, or a lower deleted pair. Then $\Ext^1(N,M) = 0$.
\end{enumerate}
\end{lem}

\begin{rem}\label{rem:diamond_factor}
    In the setting of Lemma~\ref{lem:ext}(1), Proposition~\ref{prop:factor} implies that the induced compositions $M \rightarrow E_1 \rightarrow N$ and $M \rightarrow E_2 \rightarrow N$ must both be nonzero.
\end{rem}

\begin{ex}\label{ex:homs_and_exts}
Consider the quiver $\Gamma(Q)$ in Figure~\ref{fig:AR_ex}.
    \begin{enumerate}
        \item There is a degenerate rectangular pair $\left(P(2),{\scriptsize \begin{matrix}3\\2\end{matrix}}\right)$. There is a nonzero morphism  ${P(2) \rightarrow \scriptsize \begin{matrix}3\\2\end{matrix}}$  and $\Ext^1\left({\scriptsize \begin{matrix}3\\2\end{matrix}},P(2)\right) = 0$.
        \item There is a rectangular pair $(P(4), I(4))$. There is a nonzero morphism $P(4) \rightarrow I(4)$ and a diamond short exact sequence
        $$0 \rightarrow P(4) \rightarrow I(5) \oplus 4 \rightarrow I(4).$$
        In this case, we have that $\cR_\rightarrow(P(4)) = \cR_\leftarrow(I(4))$, and this consists of the eight representations in the rectangle whose vertices are $P(4), I(4)$, $4 = \Sigma_{\nearrow}(P(4)) \cap \Sigma_{\nwarrow}(I(4))$, and $I(5) = \Sigma_{\searrow}(P(4)) \cap \Sigma_{\swarrow}(I(4))$.
        \item There is an upper deleted pair $(4, I(3))$. There is a nonsplit short exact sequence
        $$0 \rightarrow 4 \rightarrow I(4) \rightarrow I(3) \rightarrow 0$$
        and $\Hom(4, I(3)) = 0$.
    \end{enumerate}
\end{ex}

In Example~\ref{ex:homs_and_exts}(2) above, we have that $\cR_\rightarrow(P(4))$ has the shape of a rectangle in the Auslander-Reiten quiver $\Gamma(Q)$. As observed at the end of \cite[Section~3.1.4.1]{schiffler}, this happens precisely because $P(4)$ is projective. A consequence of this observation is the following.

\begin{prop}\label{prop:proj_hammock}
    Let $E_1 \in \cL_1$ and $E_2 \in \cL_n$. Suppose that there is a module $M$ at the intersection $\Sigma_{\nwarrow}(E_1) \cap \Sigma_{\swarrow}(E_2)$ and that there is a module $N$ at the intersection $\Sigma_{\nearrow}(E_1) \cap \Sigma_{\searrow}(E_2)$. Then the following hold.
    \begin{enumerate}
        \item There exists $i \in [n]$ such that $M = P(i)$ and $N = I(i)$.
        \item Either $\{M,N\} = \{E_1,E_2\}$ or there is a diamond short exact sequence
        $$0 \rightarrow M \rightarrow E_1 \oplus E_2 \rightarrow N \rightarrow 0.$$
    \end{enumerate}
\end{prop}

In words, Proposition~\ref{prop:proj_hammock} above says that any rectangle in $\Gamma(Q)$ which touches both the upper boundary $\cL_n$ and lower boundary $\cL_1$ must start at a projective and end at an injective. The following shows a sort of converse, namely that the rectangle connecting an indecomposable projective to the corresponding injective will always touch both the upper and lower boundary.

\begin{lem}\label{lem:proj_hammock_2}
    Let $i \in [n]$. Let $E_1$ and $E_2$ be the modules lying at the ends of $\Sigma_{\searrow}(P(i))$ and $\Sigma_{\nearrow}(P(i))$, respectively. Then $E_1 \in \cL_1$ and $E_2 \in \cL_n$.
\end{lem}

\begin{proof} 
   We prove only that $E_1 \in \cL_1$, the proof that $E_2 \in \cL_n$ being similar. We proceed by induction on $i$. For the base case $i = 1$, we have $E_1 = P(1) \in \cL_1$. For the induction step, let $F$ be the module lying at the end of $\Sigma_{\searrow}(P(i-1))$ and assume $F \in \cL_1$. We have two cases to consider.

    First suppose there is an irreducible morphism $P(i) \hookrightarrow P(i-1)$. Then $P(i-1)$, and thus also $F$, lies on $\Sigma_{\searrow}(P(i))$. This implies that $E_1 = F \in \cL_1$.

    Now suppose instead that there is an irreducible morphism $P(i-1) \hookrightarrow P(i)$. Thus $P(i)$ lies in the direction $(1,1)$ from $P(i-1)$. By Lemma~\ref{lem:hom}, this means
    $\Hom(P(i),F) = 0$. Then since $\Hom(P(i-1),F) \neq 0$ and there is an injection $P(i-1) \hookrightarrow P(i)$, this means $F$ is not injective.
    It follows that $E_1 = \tau^{-1} F \in \cL_1$.
    See Figure~\ref{fig:bndrylim} for an illustration.
\end{proof}

\begin{figure}
{\small
\[
\begin{tikzpicture}
[inner sep=0.5mm]
\draw (0,0) -- (2,-2);
\draw (0.25,0.25) -- (2.5,-2);
\draw (0,0) -- (0.25,0.25);
\draw [dashed] (-1,-2) -- (5,-2);
\node at (5.5,-2) {$y = 1$};
\node at (-0.75,0) {$P(i-1)$};
\node at (0.25,0.5) {$P(i)$};
 \node at (0,0) [shape=circle,draw,fill] {};
  \node at (0.25,0.25) [shape=circle,draw,fill] {};
   [shape=circle,draw,fill] {};
 \node at (2,-2) [shape=circle,draw,fill] {};
  \node at (2.5,-2) [shape=circle,draw,fill] {};
 \node at (1.75,-2.25) {$F$};
  \node at (3.25,-2.25) {$E_1 = \tau^{-1}F$};
  \end{tikzpicture}
\]}
\caption{Schematic diagram of the proof of Lemma~\ref{lem:proj_hammock_2}.}\label{fig:bndrylim}.
\end{figure}
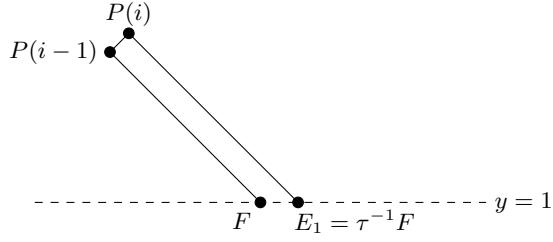

\begin{rem}\label{rem:proj_hammock}
    For $1 < i < n$ and $E_1, E_2$ as in Lemma~\ref{lem:proj_hammock_2}, the well-known description of $\cR_\rightarrow(P(i))$ implies that there is a diamond short exact sequence
    $$0 \rightarrow P(i) \rightarrow E_1 \oplus E_2 \rightarrow I(i) \rightarrow 0.$$
    See e.g. \cite[Section~3.1.4]{schiffler}.
\end{rem}

We conclude this section by tabulating the following consequence of Lemma~\ref{lem:ext}.

\begin{prop}\label{prop:rel_proj}
    Let $M$ be an indecomposable representation.
    \begin{enumerate}
        \item There exists an indecomposable $N$ such that $\Ext^1(M,N)$ contains a diamond exact sequence if and only if $M \notin \proj(KQ) \cup \cL_1 \cup \cL_n$.
        \item There exists an indecomposable $N$ such that $\Ext^1(N,M)$ contains a diamond exact sequence if and only if $M \notin \inj(KQ) \cup \cL_1 \cup \cL_n$.
    \end{enumerate}
\end{prop}

\section{Almost rigid modules}\label{sec:MAR}

In this section, we introduce an exact structure $\cE_\diamond$ on $\mods KQ$ ($Q$ a quiver of type $\mathbb{A}_n$). We then prove that the classes of almost rigid modules and basic $\cE_\diamond$-rigid modules coincide (Corollary~\ref{cor:almost rigid}). We further show that $(\mods\Lambda,\cE_\diamond)$ is a 0-Auslander category (Theorem~\ref{thm:0_auslander}), a consequence of which is the fact that the maximal almost rigid modules are precisely the $\cE_\diamond$-tilting modules (Corollary~\ref{cor:tilting}). We conclude by discussing implications for the mutation theory of maximal almost rigid modules in Section~\ref{sec:mutation}.

\subsection{The diamond exact structure}

Let $Q$ be a quiver of type $\mathbb{A}_n$ with Auslander-Reiten quiver $\Gamma(Q)$. We recall the partition $\cL_1,\ldots,\cL_n$ of the indecomposable representations introduced in Section~\ref{sec:AR_quivers}. We note that the representations in $\cL_1 \cup \cL_n$ are precisely those for which there is at most one incoming arrow and at most one outgoing arrow in $\Gamma(Q)$. We consider the following definition.

\begin{defn}\label{def:diamond_exact}
    The \emph{diamond exact structure} on $\mods KQ$ is $\cE_\diamond := \cF_{\cL_1\cup\cL_n}$.
\end{defn}

We also need the following, which is an immediate consequence of Lemma~\ref{lem:hom}.

\begin{lem}\label{lem:diamond_unique}
    Let $N$ be an indecomposable module and let $L \in \cL_1 \cup \cL_n$. Then $\Hom(L,N) \neq 0$ if and only if $L$ lies at the end of $\Sigma_{\swarrow}(N)$ or $\Sigma_{\nwarrow}(N)$.
\end{lem}

Note in particular that Lemma~\ref{lem:diamond_unique} implies that there exist at most two elements of $\cL_1 \cup \cL_n$ admitting nonzero morphisms to a given $X$. Moreover, the corresponding rectangular pairs $(L,X)$ will always be degenerate.

We now give a characterization of $\cE_\diamond(N,M)$ in the case where $M$ and $N$ are indecomposable.  As an additive bifunctor on $\mods \Lambda$, this uniquely determines the exact structure $\cE_\diamond$. More precisely, we show that the nonsplit $\cE_\diamond$-exact sequences with indecomposable endterms are precisely the diamond exact sequences. The justifies the name diamond exact structure.

\begin{prop}\label{prop:diamond_seqs}
    Let $M$ and $N$ be indecomposable. Then
   \[
\cE_\diamond(N,M) = \{ \eta \in \Ext^1(N,M) \; | \; \eta \mbox{ is split or a diamond exact sequence}\}.\]
\end{prop}

\begin{proof}
    Let $\eta = (0 \rightarrow M \rightarrow E \rightarrow N \rightarrow 0) \in \Ext^1(N,M)$  be a non-split exact sequence. We suppose first that $E$ is not indecomposable. Then, by Lemma~\ref{lem:ext}, we have $E \cong E_1 \oplus E_2$ with $E_1$ and $E_2$ both indecomposable. Schematically, the modules $E_1$ and $E_2$ lie on the corays $\Sigma_{\swarrow}(N)$ and $\Sigma_{\nwarrow}(N)$, see the right diagram of Figure~\ref{fig:seq_ex}. Now let $L \in \cL_1 \cup \cL_n$ such that $\Hom(L,N) \neq 0$. By Lemma~\ref{lem:diamond_unique}, this means $L$ lies at the end of $\Sigma_{\swarrow}(N)$ or $\Sigma_{\nwarrow}(N)$, see Figure~\ref{fig:exact}. Lemma~\ref{lem:hom} and Proposition~\ref{prop:factor} thus imply that every morphism $L \rightarrow N$ factors through the map $E_1 \oplus E_2 \rightarrow N$ comprising $\eta$, see Figure~\ref{fig:exact}. We conclude that $\eta \in \cF_{\cL_1\cup\cL_n} = \cE_\diamond$.

    For the converse, suppose that $E$ is indecomposable. By Lemma~\ref{lem:ext}, this means $(N,M)$ forms either a lower deleted pair or an upper deleted pair. We consider the case of an upper deleted pair, the other case being similar. Now let $X$ be the module which lies at the end of the coray $\Sigma_{\nwarrow}(N)$. (We have $X = \tau^{-1} L$ in the schematic on the left side of Figure~\ref{fig:seq_ex}.) By Lemma~\ref{lem:hom}, we have $\Hom(X,N) \neq 0$ and $\Hom(X,E) = 0$. Thus the sequence
     $$0 \rightarrow \Hom(X,M) \rightarrow \Hom(X,E)\rightarrow \Hom(X,N) \rightarrow 0$$
     is not exact. Since $X \in \cL_n$ by construction, we conclude that $\eta \notin \cE_\diamond$.
\end{proof}

\begin{figure}
{\small
\[
\begin{tikzpicture}
[inner sep=0.5mm]
 \draw [dashed] (6,2.75) -- (13,2.75);
\draw [dashed] (6,-3) -- (13,-3);
\node at (13.5,2.75) {$y = n$};
\node at (13.5,-3) {$y = 1$};

 \filldraw[fill=gray!50!white, draw=black] (8,1) -- (9,2) -- (12,-1) -- (11,-2) -- cycle;
 \draw (8.25,2.75) -- (9,2);
 \node at (8.25,3) {$L$};
 \node at (8.25,2.75) [shape = circle,draw,fill] {};
\node at (7.75,1) {$M$};
 \node at (8,1) [shape=circle,draw,fill] {};
 \node at (9,2) [shape=circle,draw,fill] {};
 \node at (12,-1) [shape=circle,draw,fill] {};
  \node at (11,-2) [shape=circle,draw,fill] {};
  \node at (9.25,2.25) {$E_1$};
   \node at (11,-2.25) {$E_2$};
 \node at (12.25,-1) {$N$};
  \end{tikzpicture}
\]}
\caption{Schematic diagrams of the proof of Proposition~\ref{prop:diamond_seqs}.}\label{fig:exact}
\end{figure}
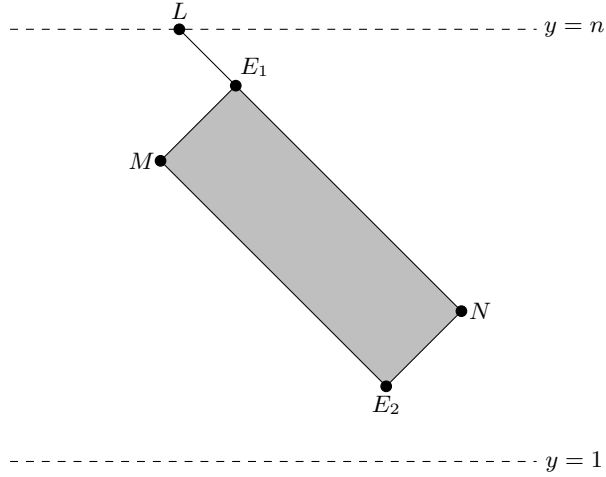

A consequence of Proposition~\ref{prop:diamond_seqs} is the following. 

\begin{coro}\label{cor:almost rigid}
Let $T \in \rep Q$. Then $T$ is almost rigid if and only if $ T$ is $\cE_\diamond$-rigid.
\end{coro}

\begin{proof}
Suppose that $T$ is almost rigid. The only nonzero extensions between two indecomposable summands of $T$ are of the form $\xi : 0\rightarrow M\rightarrow E \rightarrow N\rightarrow 0$ with $E$ indecomposable. Since $\xi \notin \cE_\diamond$, then $\cE_\diamond(T,T) = 0$.   

On the other hand, if $\cE_\diamond(T,T) = 0$, then the only non-split extensions between two indecomposable summands of $T$ are of the form $\xi : 0\rightarrow M\rightarrow E \rightarrow N\rightarrow 0$ where $E$ cannot be decomposable. Therefore, $T$ is almost rigid.
\end{proof}

We also obtain the following description of $\mathcal{E}$-projectives and $\mathcal{E}$-injectives.
\begin{coro}
    \label{cor:new}

\textup{(a)} 
    $\proj(\mathcal{E}_\diamond)=\add(\proj KQ\cup \mathcal{L}_1\cup \mathcal{L}_n)$
    
 \textup{(b)}   $\inj(\mathcal{E}_\diamond)=\add(\inj KQ\cup \mathcal{L}_1\cup \mathcal{L}_n)$

 \textup{(c)}   $\proj\textup{-}\inj(\mathcal{E}_\diamond)=\add(\mathcal{L}_1\cup \mathcal{L}_n)$
\end{coro}
\begin{proof}
    Parts (a) and (b) follow directly from Proposition~\ref{prop:exact_structure_proj}. Since the set of projective-injectives is the intersection of the two sets in (a) and (b), part (c) follows from the fact that $\proj KQ\cap \inj KQ$ is empty unless all arrows in $Q$ point in the same direction, and in that case, the only projective-injective $KQ$-module lies in $\mathcal{L}_1\cup \mathcal{L}_n.$ 
\end{proof}

\subsection{0-Auslander property}\

We now turn our attention to proving that the category $(\mods KQ, \cE_\diamond)$ is 0-Auslander.

\begin{prop}\label{prop:dominant}
    $\cE_\diamond$ has dominant dimension 1.
\end{prop}

\begin{proof}
    Let $M$ be an indecomposable $\cE_\diamond$-projective. If $M \in \cL_1 \cup \cL_n$, then $M$ is $\cE_\diamond$-injective by 
     Corollary~\ref{cor:new}, so there is nothing to show. Thus Proposition~\ref{prop:exact_structure_proj} implies that we only need to consider $M = P(i)$ for some $1 < i < n$. Now let $E_1$ and $E_2$ be the modules lying at the ends of $\Sigma_{\searrow}(M)$ and $\Sigma_{\nearrow}(M)$. By Lemma~\ref{lem:proj_hammock_2}, we have that $E_1 \in \cL_1$ and $E_2 \in \cL_n$, and so $E_1 \oplus E_2$ is $\cE_\diamond$-projective-injective by Corollary~\ref{cor:new}. Remark~\ref{rem:proj_hammock} and Proposition~\ref{prop:diamond_seqs} then say that there is an $\cE_\diamond$-exact sequence $$0 \rightarrow P(i) \rightarrow E_1 \oplus E_2 \rightarrow I(i) \rightarrow 0.$$ This proves the result.
\end{proof}

\begin{prop}\label{prop:hereditary}
    $\cE_\diamond$ has global dimension 1.
\end{prop}

\begin{proof}
    Let $M$ be an indecomposable which is not $\cE_\diamond$-projective. Equivalently, $M \notin \proj(KQ) \cup \cL_1 \cup \cL_2$,  by Corollary~\ref{cor:new}. Let $E_1$ be the module at the end of the coray $\Sigma_{\swarrow}(M)$ and let $E_2$ be the module at the end of the coray $\Sigma_{\nwarrow}(M)$. Since $M \notin \proj(KQ) \cup \cL_1 \cup \cL_2$, we have that $M$, $E_1$, and $E_2$ are all distinct. We now have two cases to consider.

    Suppose first that there is a module $N$ at 
    the intersection 
    $\Sigma_{\nwarrow}(E_1) \cap \Sigma_{\swarrow}(E_2)$. 
    In this case, $(N,M)$ is a nondegenerate rectangular pair. By Lemma~\ref{lem:ext} and Proposition~\ref{prop:diamond_seqs}, this means there is an $\cE_\diamond$-exact sequence $$0 \rightarrow N \rightarrow E_1 \oplus E_2 \rightarrow M \rightarrow 0.$$ Now if $E_1, E_2 \in \cL_1 \cup \cL_n$, then $N \in \proj(KQ)$ by Proposition~\ref{prop:proj_hammock}. Otherwise, without loss of generality we have that $E_1 \in \proj(KQ)$. Now Proposition~\ref{prop:factor} implies that the composite map $N \rightarrow E_1 \rightarrow M$ is nonzero. Since $KQ$ is hereditary, it follows that $N \in \proj(KQ)$ and thus that $N$ is $\cE_\diamond$-projective and hence $\textup{pd}_{\mathcal{E}_\diamond} \,M=1$. 

    Now suppose that there is no module at
    $\Sigma_{\nwarrow}(E_1) \cap \Sigma_{\swarrow}(E_2)$.
    We will show that this assumption leads to a contradiction.
    Choose nonzero morphisms $g_1: E_1 \rightarrow M$ and $g_2: E_2 \rightarrow M$ (these exist and are unique up to scalar multiplication by Lemma~\ref{lem:hom}), and let $p: P \rightarrow M$ be the (standard) projective cover of $M$. Then there is an exact sequence
    \begin{equation}\label{eqn:cover_redundant}\eta = (0 \rightarrow \ker f \rightarrow P \oplus E_1 \oplus E_2 \xrightarrow{f} M \rightarrow 0),\end{equation}
    where $f^\top = \begin{bmatrix} p & g_1 & g_2\end{bmatrix}$.
    (See Example~\ref{ex:degenerate_cover} for an example of why adding the direct summand $P$ is necessary.) We claim that $\eta$ is $\cE_\diamond$-exact. To see this, let $X \in \cL_1 \cup \cL_n$. By Remark~\ref{rem:FX}, it suffices to show that $\Hom(X,f)$ is surjective. We consider the case where $X \in \cL_1$, the case where $X \in \cL_n$ being similar. Now if $\Hom(X,M) = 0$ there is nothing to show. Thus suppose that $\Hom(X,M) \neq 0$. Then Lemma~\ref{lem:diamond_unique} implies that $X$ lies at the end of $\Sigma_{\swarrow}(M)$; i.e., that $X = E_1$. Then any morphism $X \rightarrow M$ must be a scalar multiple of $g_1 = f|_{E_1}$, and thus factors through $f$. We conclude that $\Hom(X,f)$ is surjective.

     Now let $N$ be an indecomposable direct summand of $\ker f$. We will show that $N$ is $\cE_\diamond$-projective.

    First note that if $\Ext^1(M,N) = 0$, then $N$ is a direct summand of $P \oplus E_1 \oplus E_2$ and we are done. Moreover, if there is a nonzero map from $N$ to $P$, then $N$ must be projective.

   It remains to consider the case where $\Ext^1(M,N) \neq 0$ and $\Hom(N,P) = 0$. Suppose that $\Hom(N,E_1) \neq 0$, the case where $\Hom(N,E_2) \neq 0$ being similar. Then Lemmas~\ref{lem:hom} and \ref{lem:ext} and Proposition~\ref{prop:factor} yield two possibilities.

   Suppose first that $\Hom(N,M) \neq 0$.
   For $i \in \{1,2\}$, let $h_i: N \rightarrow E_i$ be the morphism induced by the inclusion $N \subseteq \ker f \subseteq P \oplus E_1 \oplus E_2$. Since we have supposed $\Hom(N,M) \neq 0$ and $\Hom(N,E_1) \neq 0$, Proposition~\ref{prop:factor} implies that $g_1 \circ h_1 \neq 0$. Moreover, the assumption that $\Hom(N,P) = 0$ implies that $0 = f|_N = g_1\circ h_1 + g_2\circ h_2$. We conclude that also $g_2 \circ h_2 \neq 0$, so in particular $\Hom(N,E_2) \neq 0$. But then $(N,M)$ must be a rectangular pair with $N$ lying at the intersection $\Sigma_{\swarrow}(E_1) \cap \Sigma_{\nwarrow}(E_2)$. This is a contradiction to our assumption above.

    It remains to consider the case where $\Hom(N,M) = 0$. By Lemmas~\ref{lem:hom} and~\ref{lem:ext}, this means that $(N,M)$ is a lower deleted pair, see Figure~\ref{fig:global}. From this, we observe that there is a module at the intersection $\Sigma_{\swarrow}(E_1) \cap \Sigma_{\nwarrow}(E_2)$. This  again is a contradiction  to our assumption.
    \end{proof}

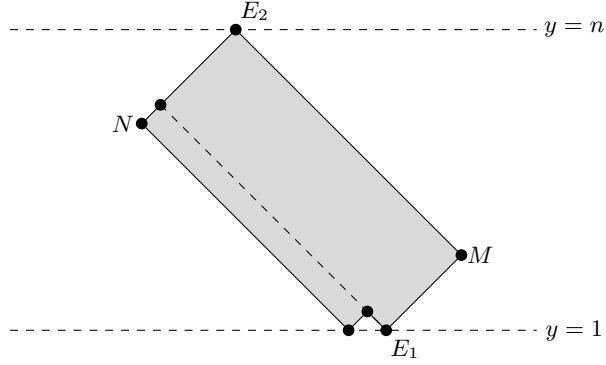
\begin{figure}
{\small
\[
\begin{tikzpicture}
[inner sep=0.5mm]
 \draw [dashed] (6,2) -- (13,2);
\draw [dashed] (6,-2) -- (13,-2);
\node at (13.5,2) {$y = n$};
\node at (13.5,-2) {$y = 1$};

 \filldraw[fill=gray!30!white, draw=black] (7.75,0.75) -- (9,2) -- (12,-1) -- (11,-2) -- (10.75,-1.75) -- (10.5,-2) -- cycle;
 \node at (8,1) [shape=circle,draw,fill] {};
\node at (7.5,0.75) {$N$};
 \node at (7.75,0.75) [shape=circle,draw,fill] {};
 \node at (9,2) [shape=circle,draw,fill] {};
 \node at (12,-1) [shape=circle,draw,fill] {};
  \node at (11,-2) [shape=circle,draw,fill] {};
  \node at (10.75,-1.75) [shape=circle,draw,fill] {};
  \node at (10.5,-2) [shape=circle,draw,fill] {};
  \node at (9.25,2.25) {$E_2$};
   \node at (11.25,-2.25) {$E_1$};
 \node at (12.25,-1) {$M$};
 \draw[dashed] (8,1) -- (11,-2);
  \end{tikzpicture}
\]}
\caption{Schematic diagram of the proof of Proposition~\ref{prop:hereditary}.}\label{fig:global}
\end{figure}

\begin{ex}\label{ex:degenerate_cover}
    Consider the representation $M = {\scriptsize \begin{matrix}1 \ 3 \ \\ \ \ \ 2 \ 4\\ \phantom{11111}5\end{matrix}}$ in Figure~\ref{fig:AR_ex}. The representation at the end of $\Sigma_{\swarrow}(M)$ is $P(1)$ and the representation at the end of $\Sigma_{\nwarrow}(M)$ is $P(5)$. In this case, we have that the morphism $P(1) \oplus P(5) \rightarrow M$ is not and $\cE_\diamond$-admissible epimorphism. (In fact, it is not even an epimorphism.) In the proof of Proposition~\ref{prop:hereditary}, we construct a $\cE_\diamond$-admissible epimorphism $P(3) \oplus P(1) \oplus P(1) \oplus P(5) \rightarrow M$ by adding the (standard) projective cover as a direct summand. Note that, while this is indeed an admissible epimorphism, it is not a minimal cover by the relative projectives of $\cE_\diamond$. That is, it is possible that the exact sequence \eqref{eqn:cover_redundant} admits a nonzero direct summand $N$ of $\ker f$ for which the induced inclusion $N \rightarrow P \oplus E_1 \oplus E_2$ is split. 
\end{ex}

\begin{thm}\label{thm:0_auslander}
    $(\mods KQ,\cE_\diamond)$ is a 0-Auslander category.
\end{thm}

\begin{proof}
    This follows from Propositions~\ref{prop:hereditary} and~\ref{prop:dominant}.
\end{proof}

\begin{coro}\label{cor:tilting}
    Let $T \in \mods KQ$. Then $T$ is maximal almost rigid if and only if $T$ is a basic $\cE_\diamond$-tilting module.
\end{coro}

\subsection{Mutation of MAR modules}\label{sec:mutation}

It was shown in \cite[Theorem E]{BGMS} that the maximal almost rigid modules form a poset isomorphic to a Cambrian lattice, where the covering relation between two MAR modules $T,T'$ is given as follows. We have $T \mathrel{<\kern-.5em\hbox{$\cdot$}} T'$
if there exist indecomposable summands $M$ of $T$ and $M'$ of $T'$ such that $T/M\cong T'/M'$ and there exists a non-split short exact sequence
\[\xymatrix{0\ar[r]&M\ar[r]^(.4)f&E\oplus E'\ar[r]&M'\ar[r]&0},\]
with $E,E'$ indecomposable summands of $T/M$ and $f$ a minimal $\add T/M$ approximation.
The unique minimal element in this lattice is the unique MAR module that contains every indecomposable projective as a direct summand, and the unique maximal element is the one that contains every indecomposable injective. 

In particular, the covering relation corresponds to the usual notion of mutation, and the Hasse diagram is the (oriented) exchange graph. 

Because of \cite[Theorem C]{BGMS}, the MAR modules are in bijection with the triangulations of a polygon with $n+1$ vertices and mutation has a combinatorial interpretation as the flip of the  diagonal corresponding to the indecomposable $M$ that is exchanged in the mutation. The boundary edges of the polygon are not mutable. They correspond to the indecomposable modules in the top and bottom $\tau$-orbits in the Auslander-Reiten quiver. These modules are precisely the indecomposable relative projective-injectives in our exact structure  $\cE_\diamond$ and therefore required in every MAR module.   
On the other hand, each of the remaining summands is mutable.

In our setting, we also recover the above mutation result from the theory of 0-Auslander algebras. Indeed, Theorem 1.3 in \cite{GNP} implies that for every MAR module $T=\overline{T}\oplus X$, with $X$ indecomposable and non-projective-injective, there exists a unique indecomposable $Y\ncong X$ such that $\overline{T}\oplus Y$ is an MAR module and there is precisely one exchange triangle 
\[\textup{either}\quad \xymatrix{X\ \ar@{>->}[r]^i&F\ar@{->>}[r]^p&Y} \quad \textup{or}\quad 
\xymatrix{Y\ \ar@{>->}[r]^{i'}&F\ar@{->>}[r]^{p'}&X}\]
where $F,F'\in \add \overline{T}$ and $i,i',p,p'$ are $\add \overline{T}$-approximations.

The order relation in \cite{BGMS} is recovered by the uniqueness of the exchange triangle. The fact that the middle term $F$ or $F'$ has at most two indecomposable summands follows from the structure of the Auslander-Reiten quiver in type $\mathbb{A}.$  To show that the middle has exactly two indecomposable summands, let us assume that $F$ is indecomposable. Then every non-trivial extension  between $X$ and $Y$ is indecomposable, because the space of extensions is at most one-dimensional in type $\mathbb{A}$. On the other hand, $T=\overline{T}\oplus X$ and  $\overline{T}\oplus Y$ are almost rigid. However, then $\overline{T}\oplus X\oplus Y$ would be almost rigid, which is a contradiction to the maximality of $T$.

    \section{Generalizations}\label{sect 5}
 \subsection{Type $\mathbb{D}$ quivers}\label{sec:typeD}
In Proposition \ref{prop:diamond_seqs} we showed that the collection of diamond exact sequences defines additive functor, which describes the exact structure  $\cE_\diamond$ on indecomposables. 
We illustrate in this section an example of a quiver of type $\mathbb{D}$ which shows that the diamond exact sequences need not be closed under pushout of morphisms,  thus they do {\em not} define an exact structure for type $\mathbb{D}$ quivers.
 We propose to replace the diamond condition in this case by requesting  three middle terms. This yields an exact structure, which we show to be not hereditary.
Note that a different generalization of almost rigid modules to type $\mathbb{D}$ were given in  the recent preprint \cite{CZ}.

\begin{ex}

 Let $Q$ be  the quiver 
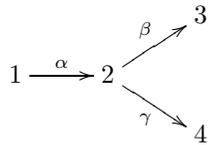
\[\xymatrix@R10pt{&&3\\ 1\ar[r]^\alpha&2\ar[ru]^\beta\ar[rd]_\gamma\\&&4}\]

Its Auslander-Reiten quiver is described as follows:

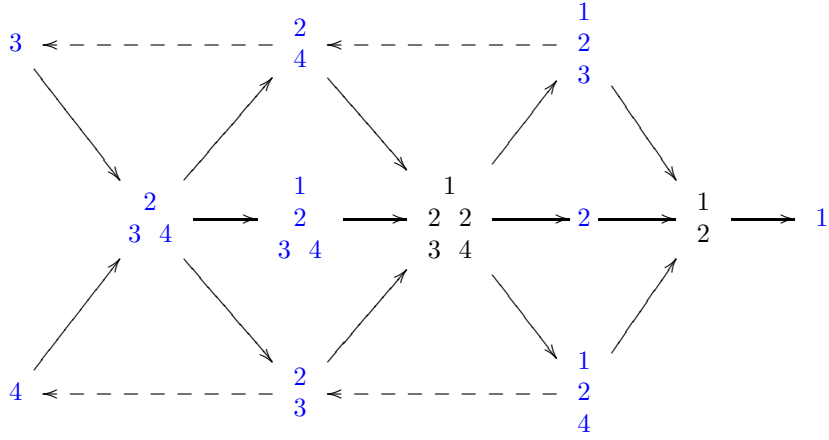
\[\xymatrix{\begin{tabular}{c} \textcolor{blue}{3} \end{tabular} \ar@{<--}[rr] \ar@{->}[dr]&&   {\begin{tabular}{c} \textcolor{blue}{2} \\ \textcolor{blue}{4} \end{tabular}}
 \ar@{->}[dr] \ar@{<--}[rr] && {\begin{tabular}{c} \textcolor{blue}{1} \\ \textcolor{blue}{2} \\ \textcolor{blue}{3} \end{tabular}} \ar@{->}[dr] \\ & {\begin{tabular}{c} \textcolor{blue}{2} \\ \textcolor{blue}{3} \ \textcolor{blue}{4}  \end{tabular}}\ar@{->}[ur] \ar@{->}[dr] \ar@{->}[r] &  {\begin{tabular}{c} \textcolor{blue}{1} \\ \textcolor{blue}{2} \\ \textcolor{blue}{3} \  \textcolor{blue}{4} \end{tabular}} \ar@{->}[r] &  {\begin{tabular}{c c} 1 \\ 2 \ 2 \\ 3 \ 4 \end{tabular}}\ar@{->}[dr] \ar@{->}[ur] \ar@{->}[r] & \textcolor{blue}{2} \ar@{->}[r] &  {\begin{tabular}{c} 1 \\ 2 \end{tabular}} \ar@{->}[r] & {\begin{tabular}{c} \textcolor{blue}{1} \end{tabular}}\\ \begin{tabular}{c} \textcolor{blue}{4} \end{tabular} \ar@{<--}[rr] \ar@{->}[ur] && {\begin{tabular}{c} \textcolor{blue}{2} \\ \textcolor{blue}{3}  \end{tabular}} \ar@{<--}[rr] \ar@{->}[ur] &&{\begin{tabular}{c} \textcolor{blue}{1} \\ \textcolor{blue}{2} \\ \textcolor{blue}{4} \end{tabular}}  \ar@{->}[ur]   }\]

We consider the diamond exact sequence $\xi :$ \[  \xymatrix{0\ar[r]&{\begin{smallmatrix}
        2\\3\ 4
    \end{smallmatrix}}\ar[r]&{\begin{smallmatrix}
        2\\4
    \end{smallmatrix}}\oplus{\begin{smallmatrix}
        1\\2\\3 \ 4
    \end{smallmatrix}}\ar[r]&{\begin{smallmatrix}
         1 \\ 2 \\ 4
    \end{smallmatrix}}\ar[r]&0.}
    \]

    By taking the push-out of $\xi$ along $f : \begin{smallmatrix}
        2\\3\ 4
    \end{smallmatrix} \rightarrow \begin{smallmatrix}
        2\\3
    \end{smallmatrix}$, we obtain the short exact sequence:
    \[  \xymatrix{0\ar[r]&{\begin{smallmatrix}
        2\\3
    \end{smallmatrix}}\ar[r]&{\begin{smallmatrix}
        1\\2 \ 2 \\3 \ 4
    \end{smallmatrix}}\ar[r]&{\begin{smallmatrix}
        1 \\ 2 \\ 4
    \end{smallmatrix}}\ar[r]&0}
    \] 
    which is not a diamond exact sequence.
    Therefore, the diamond exact sequences do not form an exact structure for type $\mathbb D$ quivers.
\bigskip

    In this case, one could consider the exact structure $\cE_{\diamonddash} := \mathcal{F}_{\mathcal{X}}$ where $\mathcal{X} ={\left\{
    {\begin{smallmatrix}
        \textcolor{blue}{3}
    \end{smallmatrix}},{\begin{smallmatrix}
        \textcolor{blue}{2} \\ \textcolor{blue}{3} \ \textcolor{blue}{4}
    \end{smallmatrix}},{\begin{smallmatrix}
        \textcolor{blue}{4}
    \end{smallmatrix}},{\begin{smallmatrix}
        \textcolor{blue}{2} \\ \textcolor{blue}{4}
    \end{smallmatrix}},{\begin{smallmatrix}
        \textcolor{blue}{2} \\ \textcolor{blue}{3}
    \end{smallmatrix}},
    {\begin{smallmatrix}
        \textcolor{blue}{1} \\ \textcolor{blue}{2} \\ \textcolor{blue}{3} \ \textcolor{blue}{4}
    \end{smallmatrix}},
    {\begin{smallmatrix}
        \textcolor{blue}{1} \\ \textcolor{blue}{2} \\ \textcolor{blue}{3}
    \end{smallmatrix}},
     {\begin{smallmatrix}
        \textcolor{blue}{1} \\ \textcolor{blue}{2} \\ \textcolor{blue}{4}
    \end{smallmatrix}},
     {\begin{smallmatrix}
        \textcolor{blue}{2}
    \end{smallmatrix}},
     {\begin{smallmatrix}
        \textcolor{blue}{1}
    \end{smallmatrix}}
    \right\}}$.

    The only Auslander-Reiten sequences in $\cE_{\diamonddash}$ are the ones having middle terms with three indecomposables:
    \[  \xymatrix{0\ar[r]&{\begin{smallmatrix}
        \textcolor{blue}{2}\\\textcolor{blue}{3}\ \textcolor{blue}{4}
    \end{smallmatrix}}\ar[r]&{\begin{smallmatrix}
        \textcolor{blue}{2}\\\textcolor{blue}{4}
    \end{smallmatrix}}\oplus{\begin{smallmatrix}
        \textcolor{blue}{1}\\\textcolor{blue}{2}\\\textcolor{blue}{3} \ \textcolor{blue}{4}
    \end{smallmatrix}}\oplus{\begin{smallmatrix}
        \textcolor{blue}{2}\\\textcolor{blue}{3} 
    \end{smallmatrix}}\ar[r]&{\begin{smallmatrix}
         1 \\ 2 \ 2 \\ 3 \ 4
    \end{smallmatrix}}\ar[r]&0}
    \] and 
    \[  \xymatrix{0\ar[r]&{\begin{smallmatrix}
        1 \\ 2 \ 2 \\ 3 \ 4
    \end{smallmatrix}}\ar[r]&{\begin{smallmatrix}
        \textcolor{blue}{1} \\ \textcolor{blue}{2}\\\textcolor{blue}{3}
    \end{smallmatrix}}\oplus{\begin{smallmatrix}
        \textcolor{blue}{2}
    \end{smallmatrix}}\oplus{\begin{smallmatrix}
         \textcolor{blue}{1} \\ \textcolor{blue}{2} \\ \textcolor{blue}{4}
    \end{smallmatrix}}\ar[r]&{\begin{smallmatrix}
         1 \\ 2 
    \end{smallmatrix}}\ar[r]&0}
    \] are in $\cE_{\diamonddash}$.

    We observe that one can construct the following $\cE_{\diamonddash}$-projective resolution of the indecomposable $\begin{smallmatrix}
         1 \\ 2 
    \end{smallmatrix}$ :

    \[  \xymatrix{0\ar[r]&{\begin{smallmatrix}
        \textcolor{blue}{2}\\\textcolor{blue}{3}\ \textcolor{blue}{4}
    \end{smallmatrix}}\ar[r]&{\begin{smallmatrix}
        \textcolor{blue}{2}\\\textcolor{blue}{4}
    \end{smallmatrix}}\oplus{\begin{smallmatrix}
        \textcolor{blue}{1}\\\textcolor{blue}{2}\\\textcolor{blue}{3} \ \textcolor{blue}{4}
    \end{smallmatrix}}\oplus{\begin{smallmatrix}
        \textcolor{blue}{2}\\\textcolor{blue}{3} 
    \end{smallmatrix}}\ar[r]&{\begin{smallmatrix}
        \textcolor{blue}{1} \\ \textcolor{blue}{2}\\\textcolor{blue}{3}
    \end{smallmatrix}}\oplus{\begin{smallmatrix}
        \textcolor{blue}{2}
    \end{smallmatrix}}\oplus{\begin{smallmatrix}
         \textcolor{blue}{1} \\ \textcolor{blue}{2} \\ \textcolor{blue}{4}
    \end{smallmatrix}}\ar[r]&{\begin{smallmatrix}
         1 \\ 2 
    \end{smallmatrix}}\ar[r]& 0}
    \] 

    Therefore, $\pd_{\cE_{\diamonddash}}(\begin{smallmatrix}
         1 \\ 2 
    \end{smallmatrix}) = 2 $ and so $(\rep Q,\cE_{\diamonddash})$ is not a \emph{0-Auslander category}. 
It is easy to see that
 ${\begin{smallmatrix}
        \textcolor{blue}{3}
\end{smallmatrix}}\oplus{\begin{smallmatrix}
        \textcolor{blue}{2} \\ \textcolor{blue}{3} \ \textcolor{blue}{4}
    \end{smallmatrix}}\oplus
    {\begin{smallmatrix}
        \textcolor{blue}{4}
    \end{smallmatrix}}\oplus{\begin{smallmatrix}
        \textcolor{blue}{2} \\ \textcolor{blue}{4}
    \end{smallmatrix}}\oplus{\begin{smallmatrix}
        \textcolor{blue}{2} \\ \textcolor{blue}{3}
    \end{smallmatrix}}\oplus
    {\begin{smallmatrix}
        \textcolor{blue}{1} \\ \textcolor{blue}{2} \\ \textcolor{blue}{3} \ \textcolor{blue}{4}
    \end{smallmatrix}}\oplus
    {\begin{smallmatrix}
        \textcolor{blue}{1} \\ \textcolor{blue}{2} \\ \textcolor{blue}{3}
    \end{smallmatrix}}\oplus
     {\begin{smallmatrix}
        \textcolor{blue}{1} \\ \textcolor{blue}{2} \\ \textcolor{blue}{4}
    \end{smallmatrix}}\oplus
     {\begin{smallmatrix}
        \textcolor{blue}{2}
    \end{smallmatrix}}\oplus
     {\begin{smallmatrix}
        \textcolor{blue}{1}
    \end{smallmatrix}}\oplus{\begin{smallmatrix}
       1 \\ 2
    \end{smallmatrix}}$ is \emph{maximal almost rigid} in the sense of Definition~\ref{def:almost_rigid} (replacing diamond exact by having three middle terms),
    but it is \emph{not} $\cE_{\diamonddash}$-tilting since it contains a module of $\cE_{\diamonddash}$-projective dimension 2.
\end{ex}

\subsection{Gentle algebras}\label{sec:gentle} 
In this section, we give an example that illustrates 
how our results may generalize to the much more general setting of gentle algebras.

 It is shown in \cite{OPS,BCS} that gentle algebras are in bijection with surface dissections. In \cite{BCSGS}, the concept of maximal almost rigid modules is generalized to gentle algebras, and the authors prove that many of the results generalize from type $\mathbb{A}$ to the gentle setting. For example the maximal almost rigid modules over a gentle algebra are in bijection with permissible triangulations of its surface 
\cite[Theorem 1.2]{BCSGS}.

For the gentle algebra in the following example  our Theorem \ref{thm:intro} holds.
\begin{ex}
    Let $A$ be given by the quiver 
    
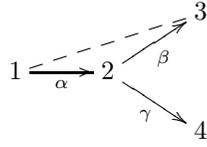
\[\xymatrix@R10pt{&&3\\ 1\ar[r]_\alpha \ar@{--}[rru]&2\ar[ru]_\beta\ar[rd]_\gamma\\&&4}\]
    bound by the relation $\alpha\beta.$
    Its Auslander-Reiten quiver is 
    
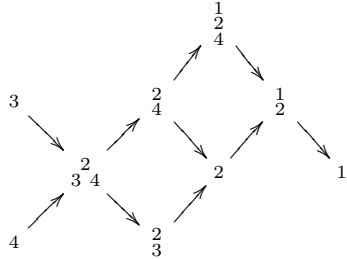
\[\xymatrix@R10pt@C10pt{
    &&&{\begin{smallmatrix}
        1\\2\\4
    \end{smallmatrix}}\ar[rd]
    \\
    {\begin{smallmatrix}
        3
    \end{smallmatrix}}\ar[rd] &&
    {\begin{smallmatrix}
        2\\4
    \end{smallmatrix}}\ar[rd]\ar[ru] &&
    {\begin{smallmatrix}
        1\\2
    \end{smallmatrix}}\ar[rd] 
    \\
    &{\begin{smallmatrix}
        2\\3\ 4
    \end{smallmatrix}}\ar[rd]\ar[ru] &&
    {\begin{smallmatrix}
        2
    \end{smallmatrix}}\ar[ru] &&
    {\begin{smallmatrix}
        1
    \end{smallmatrix}}
    \\
    {\begin{smallmatrix}
        4
    \end{smallmatrix}}\ar[ru] &&
    {\begin{smallmatrix}
        2\\3
    \end{smallmatrix}}\ar[ru] &&
    }
    \]
\end{ex}
Let $\mathcal{X}=\left\{{\begin{smallmatrix}
        3
    \end{smallmatrix}},
    {\begin{smallmatrix}
        4
    \end{smallmatrix}},{\begin{smallmatrix}
        2\\3\ 4
    \end{smallmatrix}},{\begin{smallmatrix}
        1\\2\\4
    \end{smallmatrix}},{\begin{smallmatrix}
        2\\4
    \end{smallmatrix}},{\begin{smallmatrix}
        2\\3
    \end{smallmatrix}},
    {\begin{smallmatrix}
        1
    \end{smallmatrix}}
    \right\}$. Then the direct sum of the elements of $\mathcal{X}$ is a maximal almost rigid $A$-module.

    Since the are only finitely many short exact sequences with indecomposable endterms, we can check by hand that  $\mathcal{F}_{\mathcal{X}}$ contains only short exact sequences with decomposable middle term. Moreover, it is easy to see that the relative global dimension of $\textup{mod}\,A$ is 1. Indeed the only indecomposables that are not $\mathcal{F}_{\mathcal{X}}$-projective are the modules ${\begin{smallmatrix}
        2
    \end{smallmatrix}}$
    and ${\begin{smallmatrix}
        1\\2
    \end{smallmatrix}}$
    and their $\mathcal{F}_{\mathcal{X}}$-projective resolutions are 
    \[  \xymatrix{0\ar[r]&{\begin{smallmatrix}
       2\\3\ 4
    \end{smallmatrix}}\ar[r]&{\begin{smallmatrix}
        2\\4
    \end{smallmatrix}}\oplus 
    {\begin{smallmatrix}
        2\\3
    \end{smallmatrix}}\ar[r]&{\begin{smallmatrix}
        2
    \end{smallmatrix}}\ar[r]&0} 
    \qquad \textup{and} \qquad
    \xymatrix{0\ar[r]&{\begin{smallmatrix}
          2\\3\ 4
    \end{smallmatrix}}\ar[r]&{\begin{smallmatrix}
        1\\2\\4
    \end{smallmatrix}}\oplus 
    {\begin{smallmatrix}
        2\\3
    \end{smallmatrix}}\ar[r]&{\begin{smallmatrix}
        1\\2
    \end{smallmatrix}}\ar[r]&0}.
    \]
    Furthermore, the relative projective-injective modules are
    $\left\{{\begin{smallmatrix}
        3
    \end{smallmatrix}},
    {\begin{smallmatrix}
        4
    \end{smallmatrix}},{\begin{smallmatrix}
        2\\3
    \end{smallmatrix}},{\begin{smallmatrix}
        1\\2\\4
    \end{smallmatrix}},
    {\begin{smallmatrix}
        1
    \end{smallmatrix}}
    \right\}$.

    However, the category is not 0-Auslander. Indeed the only $\mathcal{F}_{\mathcal{X}}$-projectives that are not $\mathcal{F}_{\mathcal{X}}$-projective-injective are ${\begin{smallmatrix}
        2\\3\ 4
    \end{smallmatrix}}$ and $
    {\begin{smallmatrix}
        2\\4
    \end{smallmatrix}},$ and these modules admit the following short exact sequences 
\[  \xymatrix{0\ar[r]&{\begin{smallmatrix}
        2\\3 \ 4
    \end{smallmatrix}}\ar[r]&{\begin{smallmatrix}
        2\\3
    \end{smallmatrix}}\oplus{\begin{smallmatrix}
        1\\2\\4
    \end{smallmatrix}}\ar[r]&{\begin{smallmatrix}
        1\\2
    \end{smallmatrix}}\ar[r]&0} 
    \qquad \textup{and} \qquad
    \xymatrix{0\ar[r]&{\begin{smallmatrix}
        2\\4
    \end{smallmatrix}}\ar[r]&{\begin{smallmatrix}
        1\\2\\4
    \end{smallmatrix}}\ar[r]&{\begin{smallmatrix}
        1
    \end{smallmatrix}}\ar[r]&0}
    \]
    whose middle terms are  $\mathcal{F}_{\mathcal{X}}$-projective-injective and whose endterms are $\mathcal{F}_{\mathcal{X}}$-injective. The problem is that the last sequence is not $\mathcal{F}_{\mathcal{X}}$-admissible, because $S(1)\in \mathcal{X}$ but $ \Hom(S(1),-)$ is not exact on this sequence. Thus this category is not 0-Auslander.

\end{document}